\documentclass[a4paper,11pt]{amsart}

\pdfoutput=1

\usepackage[text={400pt,660pt},centering]{geometry}

\usepackage{amsthm, amssymb, amsmath, amsfonts, mathrsfs, bbm}
\usepackage{mathtools}
\usepackage{scalerel} 
\usepackage{stackrel}
\usepackage{comment}
\usepackage{a4wide}
\usepackage{float}

\usepackage[colorlinks=true, pdfstartview=FitV, linkcolor=blue, citecolor=blue, urlcolor=blue,pagebackref=false]{hyperref}
\usepackage{esint} 
\usepackage{MnSymbol} 

\usepackage{microtype}
\usepackage{cleveref}

\newtheorem{proposition}{Proposition}
\newtheorem{theorem}[proposition]{Theorem}
\newtheorem{lemma}[proposition]{Lemma}

\theoremstyle{remark}
\newtheorem{remark}[proposition]{Remark}

\theoremstyle{definition}

\numberwithin{equation}{section}
\numberwithin{proposition}{section}
\numberwithin{table}{section}

\renewcommand{\leq}{\leqslant}
\renewcommand{\geq}{\geqslant}

\renewcommand{\epsilon}{\varepsilon}

\title[Absence of weak disorder for
polymers on percolation clusters]{Absence of weak disorder for directed
polymers on supercritical percolation clusters}
\author[M.\ Nitzschner]{
Maximilian Nitzschner}
\address{Department of Mathematics, The Hong Kong University of Science and Technology}
\curraddr{Clear Water Bay, Kowloon, Hong Kong}
\email{mnitzschner@ust.hk}

\begin{document}

\begin{abstract}

We study the directed polymer model on infinite clusters of supercritical Bernoulli percolation containing the origin in dimensions $d \geq 3$, and prove that for almost every realization of the cluster and every strictly positive value of the inverse temperature $\beta$, the polymer is in a strong disorder phase, answering a question from Cosco, Seroussi, and Zeitouni~\cite{cosco2021directed}.

\bigskip


\medskip


\end{abstract}
\maketitle

\tableofcontents

%
%
%
%
%
%
%
%
\section{Introduction}

In this article, we study the model of a directed polymer in a random environment in which the underlying graph consists of a supercritical cluster of Bernoulli (bond) percolation on $\mathbb{Z}^d$ with $d \geq 3$. \medskip

The model of directed polymers on $\mathbb{Z}^d$ has a rich history both in the physics and mathematics literature, and we refer to~\cite{comets2017directed,den2009random} for a thorough introduction. The case of directed polymers on other graphs beyond $\mathbb{Z}^d$ (such as downward paths on trees, diffusions on the $d$-dimensional discrete torus or on the cylinder, or simple random walk on the complete graph) has also been studied (see, e.g.,~\cite{brunet2000probability,buffet1993directed,
comets2019random,derrida1988polymers,eckmann1989largest}).
Motivated by the work~\cite{seroussi2018spectral} on stochastic dynamics in large, finite networks, in the recent article~\cite{cosco2021directed}, the authors consider polymer models on general infinite graphs, which includes results on random graphs such as percolation clusters or Galton-Watson trees, with an emphasis on new phenomena emerging for polymer models when the underlying graph structure is modified. We also refer to~\cite{carmona2006strong} for results on polymers where the underlying walk is a general Markov chain, and to~\cite{kajino2020two} concerning polymers on strongly recurrent graphs similar to those in~\cite{cosco2021directed}. In the present work, we investigate a polymer in a random environment defined on a typical realization of an infinite cluster in Bernoulli (bond) percolation on $\mathbb{Z}^d$ for $p \in (p_c(d),1)$ (with $p_c(d)$ the percolation threshold), containing the origin. As our main contribution, we establish that for almost every realization of such a percolation cluster in $d \geq 3$, there is no non-trivial weak disorder regime for the polymer, answering a question from~\cite[Section 7]{cosco2021directed}. \medskip

We now describe our model and main result in more detail. Denote by $E(\mathbb{Z}^d)$ the set of nearest-neighbor edges in the Euclidean lattice and consider for $p \in [0,1]$ a probability measure $\mathbb{Q}$ on $\Omega_0 = \{0,1 \}^{E(\mathbb{Z}^d)}$ endowed with its canonical $\sigma$-algebra $\mathcal{A}$ such that the canonical coordinates $(\mu_e)_{e \in  E(\mathbb{Z}^d)}$ are i.i.d.~Bernoulli random variables with parameter $p$. We let $p_c(d)$ stand for the critical parameter for bond percolation, and recall that for $p > p_c(d)$, there is a $\mathbb{Q}$-a.s.~unique infinite connected component of open bonds denoted by $\mathscr{C}_\infty$ (see~\cite{grimmett1999percolation}). Throughout the remainder of this article, we will assume that
\begin{equation}
p_c(d) < p < 1,
\end{equation}
and we will often consider
\begin{equation}
\label{eq:Q0Def}
\mathbb{Q}_0 = \mathbb{Q}\left[ \cdot \, | \, 0 \in \mathscr{C}_\infty \right],
\end{equation}
(where for $z \in \mathbb{Z}^d$, $\{z \in \mathscr{C}_\infty \} \in \mathcal{A}$ is the event of positive $\mathbb{Q}$-probability that there exists a unique infinite open cluster containing the vertex $z$). The probability of the event $\{0 \in \mathscr{C}_\infty\}$ will be denoted by $\theta(p)$. For $\mu \in \{0 \in \mathscr{C}_\infty \}$, we can consider the discrete-time simple random walk $(X_k)_{k \geq 0}$ on the percolation cluster $\mathscr{C}_\infty$ started at the origin, and denote the corresponding governing probability measure by $P_{0,\mu}$ (we refer to Section~\ref{s.notation} for precise definitions). To define the polymer measure, we consider i.i.d.~random variables
\begin{equation}
\label{eq:iidEnvironment}
\{\omega(i,x) \, : \, i \in \mathbb{N}_0, x \in \mathbb{Z}^d \}, 
\end{equation}
with mean zero and unit variance, governed by a probability measure $\mathbb{P}$, called the \textit{environment measure}. We also assume that these random variables possess exponential moments, see~\eqref{eq:LambdaDef}, and we define the $\sigma$-algebra $\mathcal{G}_n = \sigma(\omega(i,x) \, : \, i \leq n, x \in \mathbb{Z}^d)$, for $n \in \mathbb{N}_0$. The \textit{polymer measure} of time horizon $n \in \mathbb{N}_0$, inverse temperature $\beta \geq 0$, and percolation configuration $\mu \in \{0 \in \mathscr{C}_\infty \}$ is a probability measure defined on paths $X = (X_k)_{k \geq 0}$ by
\begin{equation}
\mathrm{d}P_{0,\mu}^{n,\beta}[X] = \frac{1}{Z_{n,\mu}(\beta,\omega)} \exp\left(\beta \sum_{i = 1}^n \omega(i,X_i) \right) \mathrm{d}P_{0,\mu}[X], 
\end{equation}
with the \textit{partition function}
\begin{equation}
Z_{n,\mu}(\beta,\omega) = E_{0,\mu}\left[\exp\left(\beta \sum_{i = 1}^n \omega(i,X_i) \right)  \right],
\end{equation}
(here $E_{0,\mu}$ stands for the expectation under the probability measure $P_{0,\mu}$). Various methods exist to characterize the asymptotic behavior of the directed polymer for large $n$. A key role in this context is played by the \textit{normalized partition function} 
\begin{equation}
\label{eq:NormPartFct}
W_{n,\mu}  = \frac{Z_{n,\mu}}{\mathbb{E}[Z_{n,\mu}]},
\end{equation}
which is easily shown to be a non-negative, mean-one $(\mathcal{G}_n)_{n \geq 0}$-martingale and therefore converges $\mathbb{P}$-a.s.~to a limit $W_{\infty,\mu}$ (note that $W_{n,\mu}$ and $W_{\infty,\mu}$ also depend on $\beta$ and $\omega$, but we suppress this dependence in the notation). In a more general context of locally finite graphs, it has been established in~\cite[Proposition 1.2]{cosco2021directed} that there exists $\beta_c$ ( $= \beta_c(\mu)$, for the situation considered here) $\in [0,\infty]$ such that
\begin{equation}
\begin{split}
W_{\infty,\mu} > 0, \qquad \mathbb{P}\text{-a.s.} & \text{ for all }\beta < \beta_c \qquad (\textit{weak disorder}), \\
W_{\infty,\mu} = 0, \qquad \mathbb{P}\text{-a.s.} & \text{ for all }\beta > \beta_c \qquad (\textit{strong disorder}).
\end{split}
\end{equation}
For the model of directed polymers associated to a simple random walk on the full lattice $\mathbb{Z}^d$ (corresponding informally to $p = 1$), it is known that $\beta_c = 0$ in $d = 1,2$ and $\beta_c > 0$ for $d \geq 3$, see~\cite{comets2017directed,lacoin2010new}. Crucially, the weak disorder phase essentially corresponds to a region of \textit{delocalization} for the directed polymer, and assuming uniform integrability for the normalized partition function (a condition slightly stronger than weak disorder), one can prove an almost sure central limit theorem for the polymer path, see~\cite{comets2006directed} for the result on $\mathbb{Z}^d$ (in this case, the result in fact holds in the entire weak disorder region), and~\cite[Theorem 1.3]{cosco2021directed} for its adaptation to general graphs. On the other hand, strong disorder implies certain localization results for the polymer path and we refer to~\cite{carmona2002partition,comets2003directed} for statements on $\mathbb{Z}^d$ and~\cite[Theorem 1.6]{cosco2021directed} for their adaptation to general graphs. In our main result, we show that in contrast to directed polymers on a high-dimensional full lattice, the critical inverse temperature vanishes, for $\mathbb{Q}_0$-a.e.~realization of $\mu$.

\begin{theorem}
\label{thm:MainTheorem}
Let $d \geq 3$ and $p \in (p_c(d),1)$. Then, for $\mathbb{Q}_0$-a.e.~realization of $\mu \in \{0 \in \mathscr{C}_\infty \}$, one has
\begin{equation}
\beta_c(\mu) = 0.
\end{equation}
\end{theorem}

This answers question (3.) from~\cite[Section 7]{cosco2021directed}, and we remark that for $d = 2$, the fact that $\beta_c(\mu) = 0$ for $\mathbb{Q}_0$-a.e.~realization of $\mu \in \{0 \in \mathscr{C}_\infty\}$ was already proved in the same reference. To give some context, let us point out that another critical parameter $\beta_2 \in [0,\infty]$ corresponding to a region $0 \leq \beta < \beta_2$ known as \textit{$L^2$-phase} is often introduced for general polymer models. Essentially, there one requires that the martingales $(W_{n,\mu}(\beta))_n$ are uniformly bounded in $L^2$, which facilitates certain moment calculations. It is immediate that $\beta_2 \leq \beta_c$, and in fact one has $0 < \beta_2 < \beta_c$ on $\mathbb{Z}^d$ in dimensions $d \geq 3$, see~\cite{berger2010critical,birkner2011collision,birkner2010annealed,birkner2011disorder}. As a result of~\cite[Theorem 5.5]{cosco2021directed}, $\beta_2(\mu) = 0$ for $\mathbb{Q}_0$-a.e.~realization of the the supercritical percolation cluster $\mu$ on $\mathbb{Z}^d$, $d \geq 2$. Note however that there are graphs for which one has $0 = \beta_2 < \beta_c$, see~\cite[Theorem 5.7]{cosco2021directed}. \medskip

Let us briefly comment on the proof of Theorem~\ref{thm:MainTheorem}. The main point is to establish that for every $\beta > 0$ and $\mathbb{Q}_0$-a.e.~realization of $\mu$, the martingale $(W_{n,\mu}(\beta))_{n}$ converges to $0$. Similarly as in~\cite[Section 4.3]{cosco2021directed}, we utilize a \textit{fractional moment} calculation based on a change of measure argument introduced in~\cite{lacoin2010new}. A central challenge is to find regions in space-time which are typically visited by the unconstrained underlying random walk, that allow for an effective implementation of the aforementioned change of the environment measure in these regions at a ``relatively low cost''. In our context, these regions will be certain \textit{tubes} of length proportional to $\log n$, meaning straight lines of open edges with all neighboring perpendicular edges closed, which typically appear frequently in a box of size $O(n^\gamma)$ (for any $\gamma > 0$) and are visited by the walk before time $n$. While for the proof of $\beta_2(\mu) = 0$ in~\cite{cosco2021directed}, the existence of such large tubes somewhere in the cluster was sufficient, in the present context we require quantitative bounds that show that the random walk on the percolation cluster tends to enter such a tube with high probability. The last property requires the use of some quenched heat kernel bounds on the infinite cluster to enter such tubes, and more importantly a control on the probability of their existence in the cluster, which is the main technical part in this article. For this, we utilize the methods developed in~\cite[Appendix A]{dario2021quantitative} in the context of quantitative homogenization of the Green's function on the infinite cluster of supercritical Bernoulli percolation. \medskip

The remainder of this article is organized as follows: In Section~\ref{s.notation}, we introduce further notation and state some useful known results, in particular on Bernoulli percolation and quenched heat kernel bounds for the random walk on the infinite cluster. In Section~\ref{s.FracMoments}, we utilize the fractional moment calculation alluded to above to prove Theorem~\ref{thm:MainTheorem}, conditional on the fact that the random walk enters some tube of length $\varepsilon \log(n)$ and takes at least $(\varepsilon \log(n))^3$ steps there before time $n$. This last part is formulated in Proposition~\ref{prop:NoTubeAvoidance}, and we present its proof in the final Section~\ref{s.ProofMainProp}. \medskip

Finally, we state the convention used for constants. By $C,c,c',...$ we denote positive constants depending only on $d$ which may change from place to place. Numbered constants such as $c_1,c_2,...$ refer to the value assigned to them at their first appearance in the text. Dependence of constants on additional parameters will be mentioned explicitly in the notation.

%
%
%
%
%
%
%
%
\section{Notation and useful results}
\label{s.notation}

In this section we introduce further notation used in the remainder of the article and collect several useful results concerning the random walk on $\mathbb{Z}$ and on the infinite cluster of supercritical Bernoulli percolation, mainly for later use in Section~\ref{s.ProofMainProp}. We also include proofs of some standard bounds for completeness. Throughout the article, unless stated otherwise, we tacitly assume $d \geq 3$. \medskip

We start with some elementary notation. We use the convention $\mathbb{N}_0 = \{0,1,2,...\}$ for the set of non-negative integers and set $\mathbb{N} = \mathbb{N}_0 \setminus \{0\}$. As in the introduction, we let $\mathbb{Z}^d$ stand for the $d$-dimensional integer lattice and we write $\mathbb{Z}^d_o$ resp.~$\mathbb{Z}^d_e$ for the points $x \in \mathbb{Z}^d$ such that $\sum_{i = 1}^n x_i$ is odd resp.~even. For real numbers $s,t$, $s \vee t$ and $s \wedge t$ denote the maximum and minimum of $s$ and $t$, respectively, and we denote the integer part of $s$ by $[s]$. We also write $s_+ = s \vee 0$ for the positive part of $s$. We denote by $| \, \cdot \, |$, $| \, \cdot \, |_1$, and $| \, \cdot \, |_\infty$ the Euclidean, $\ell^1$- and $\ell^\infty$-norms on $\mathbb{R}^d$, respectively. If $x,y \in \mathbb{Z}^d$ fulfill $|x-y|_1 = 1$, we say that $x$ and $y$ are adjacent (or nearest neighbors) and write $x \sim y$. Unordered pairs $\{x,y\}$ of vertices $x,y \in \mathbb{Z}^d$ are called edges if $x \sim y$, and denote the set of all edges by $E(\mathbb{Z}^d)$. For a set $K \subseteq \mathbb{Z}^d$, we let $E(K)$ stand for the set of edges $\{x,y\} \in E(\mathbb{Z}^d)$ with $K \cap \{x,y\} \neq \emptyset$. For $x \in \mathbb{Z}^d$ and $r \geq 0$, we let $B(x,r) = \{y \in \mathbb{Z}^d \, : \, |x-y|_\infty \leq r\} \subseteq \mathbb{Z}^d$ stand for the closed $\ell^\infty$-ball of radius $r$ and center $x$. Occasionally we also need the sets $B_o(x,r) = B(x,r) \cap \mathbb{Z}^d_o$ resp.~$B_e(x,r) = B(x,r) \cap \mathbb{Z}^d_e$, consisting of the points in the box $B(x,r)$ with odd resp.~even parity. The cardinality of a set $K \subseteq \mathbb{Z}^d$ is denoted by $|K|$. The $\ell^1$-distance between two sets $U_1,U_2 \subseteq \mathbb{Z}^d$ is defined by $d_{\ell^1}(U_1,U_2) = \inf\{|x - y|_1 \, : \, x \in U_1, y \in U_2\}$ (with the convention $\inf \emptyset = \infty$). \medskip

We turn to some more notation concerning the environment attached to the directed polymer model. Recall the notation~\eqref{eq:iidEnvironment} for the i.i.d.~random variables $\omega(i,x)$ under $\mathbb{P}$ indexed by $\mathbb{N}_0 \times \mathbb{Z}^d$, and the convention that
\begin{equation}
\label{eq:AssumptionsExpVar}
\mathbb{E}[\omega(i,x)] = 0, \qquad \mathbb{E}[\omega(i,x)^2] = 1, \qquad (i,x) \in \mathbb{N}_0 \times \mathbb{Z}^d,
\end{equation}
where $\mathbb{E}$ stands for the expectation under $\mathbb{P}$. We use the notation 
\begin{equation}
\label{eq:LambdaDef}
\Lambda(\beta) = \log \mathbb{E}[\exp(\beta \omega(i,x))], \qquad \beta \in \mathbb{R},
\end{equation}
and assume throughout the remainder of the article that $\Lambda(\beta)$ is finite for all $\beta \geq -a$ for some $a > 0$. With this definition, one can bring the normalized partition function introduced for $\mu \in \{0 \in \mathscr{C}_\infty\}$ in~\eqref{eq:NormPartFct} into the form
\begin{equation}
\label{eq:NormPartFctRewritten}
W_{n,\mu} = Z_{n,\mu} \exp\left( - n\Lambda(\beta) \right), \qquad n \in \mathbb{N}_0, \beta \geq 0,
\end{equation}
(where we again suppress the dependence on $\beta$ and $\omega$ for notational convenience). For later use, we record that by our assumptions~\eqref{eq:AssumptionsExpVar} on the mean and the variance of $\omega(i,x)$ and on $\Lambda$ below~\eqref{eq:LambdaDef}, one has for sufficiently small $\delta > 0$ that
\begin{equation}
\label{eq:UsefulExpressionLambda}
\exp(\Lambda(-\delta)) = \mathbb{E}\left[\exp(-\delta \omega(i,x) )\right] = \exp\left(c_\delta \frac{\delta^2}{2}\right), \qquad c_\delta = 1 + o_\delta(1) \text{ as $\delta \rightarrow 0$}.
\end{equation}

\medskip

We now turn to the random walk on $\mathbb{Z}$ as well as on the infinite connected component of the supercritical Bernoulli percolation cluster. We start with the former by recording a simple lower bound on an upward deviation of the exit time from a symmetric interval. This bound will be used in Section~\ref{s.ProofMainProp} to control the probability that a simple random walk on the supercritical cluster spends an atypically large time inside an ``open tube'' (which is effectively one-dimensional) upon visiting its center. Its simple proof is given for completeness.
\begin{lemma}
\label{lem:DeviationExitTime}
Let $(X_n)_{n \geq 0}$ stand for the symmetric simple random walk on $\mathbb{Z}$ starting in $x \in \mathbb{Z}$ governed by $P_x^{\mathbb{Z}}$ and let $\tau_{K} = \inf\{ n \in \mathbb{N}_0 \, : \, Z_n \notin \{-K,...,K\} \}$ denote the exit time of $(X_n)_{n \geq 0}$ from the interval $[-K,K] \cap \mathbb{Z}$, with $K \in \mathbb{N}$, $K \geq 10$. Then 
\begin{equation}
P_0^{\mathbb{Z}}[\tau_K \geq K^3] \geq \exp(-cK).
\end{equation}
\end{lemma}
\begin{proof}
We denote by $p^{\ast,K}_n(x,y) = P_x^{\mathbb{Z}}[X_n =y, \tau_K > n]$ the heat kernel for the simple random walk killed upon exiting $[-K,K] \cap \mathbb{Z}$, which fulfills 
\begin{equation}
p^{\ast,K}_n(x,y) + p^{\ast,K}_{n+1}(x,y) \geq \frac{c}{\sqrt{n}} \exp\left( - \frac{c'|x-y|^2}{n} \right) 
\end{equation}
for $x,y \in [-[K/2], [K/2]] \cap \mathbb{Z}$ and $|x-y| \leq n \leq K^2$ (see, e.g.,~\cite[Theorem 4.25]{barlow2017random}). In particular, for every $x \in [-[K/2], [K/2]] \cap \mathbb{Z}$, we have $P_x[X_{K^2} \in [-[K/2], [K/2]] \cap \mathbb{Z} , \tau_K > K^2] \geq c$. The claim then follows by a $(K-1)$-fold application of the simple Markov property at times $\{K^2,2K^2,...,K^3 - K^2\}$.
\end{proof}

Next, we turn to the random walk on $\mathscr{C}_\infty$ and introduce some further notation. Recall the notation $\mathbb{Q}$ for the probability measure on $(\{0,1\}^{E(\mathbb{Z}^d)}, \mathcal{A})$ governing the Bernoulli percolation from the introduction as well as the conditional measure $\mathbb{Q}_0$ in~\eqref{eq:Q0Def}. For a configuration $\mu \in \{0,1 \}^{E(\mathbb{Z}^d)}$, we consider the discrete-time Markov chain on $\mathbb{Z}^d$ with jump probabilities $r_\mu(x,y)$ for $x,y \in \mathbb{Z}^d$ given by
\begin{equation}
\begin{split}
r_\mu(x,x) & = 1, \text{ if }\mu(e) = 0 \text{ for all } e \in E(\mathbb{Z}^d) \text{ with } e \cap \{ x\} \neq \emptyset, \\
r_\mu(x,y) & = \frac{1}{\sum_{z \in \mathbb{Z}^d \, : \, z \sim x} \mu(\{x,z\}) }, \text{ if } x \sim y \text{ and } \mu(\{x,y\}) = 1, \\
r_\mu(x,y) & = 0, \text{ otherwise}.
\end{split} 
\end{equation}
We let $P_{x,\mu}$ stand for the canonical law on $(\mathbb{Z}^d)^{\mathbb{N}_0}$ of the chain started at $x \in \mathbb{Z}^d$ and denote the canonical coordinates by $(X_n)_{n \geq 0}$. For $\mu \in \{x \in \mathscr{C}_\infty \}$, the process $(X_n)_{n \geq 0}$ under $P_{x,\mu}$ is the discrete-time simple random walk on the unique infinite cluster $\mathscr{C}_\infty$ starting in $x$.  \medskip 

We now note some known results from~\cite{barlow2004random} concerning heat kernel bounds for the random walk on the infinite cluster of supercritical Bernoulli percolation for later use. There exists a set $\Omega_1 \in \mathcal{A}$ of full $\mathbb{Q}$-measure and random variables $(R_x)_{x \in \mathbb{Z}^d}$ with 
\begin{equation}
R_x < \infty, \qquad \text{ for }\mu \in \Omega_1, x \in \mathscr{C}_\infty,
\end{equation}
such that for $\mu \in \Omega_1$, $x,y \in \mathscr{C}_\infty$, and $n \geq R_x(\mu) \vee |x-y|_1$ one has the following Gaussian lower bound on the heat kernel:
\begin{equation}
\label{eq:HKLB}
\begin{split}
P_{x,\mu}[X_n = y \text{ or } X_{n+1} = y] & \geq \frac{c}{n^{\frac{d}{2}}} \exp\left(-\frac{c'|x-y|^2}{n} \right)
\end{split}
\end{equation}
(we use the convention $R_x(\mu) = \infty$ if $\mu \notin \Omega_1$, or $\mu \in \Omega_1$ but $x \notin \mathscr{C}_\infty$). Furthermore, we have for $x \in \mathbb{Z}^d$, $n \geq 0$, the bound
\begin{equation}
\mathbb{Q}[x \in \mathscr{C}_\infty, R_x \geq n] \leq c \exp\left(-c' n^{c_1} \right).
\end{equation}
This immediately implies the following observation.
\begin{lemma}
\label{lem:ShortMixingTimesCube}
Let $\gamma > 0$. For $\mathbb{Q}$-a.e.~$\mu \in \Omega_0$, there exists $N_0^{\textnormal{reg}}(\mu) < \infty$ such that for every $n \geq N_0^{\textnormal{reg}}(\mu)$ and $x \in \mathscr{C}_\infty \cap B(0,n)$ one has that $R_x < n^\gamma$.
\end{lemma}
\begin{proof}
Consider the events $D_n = \bigcup_{ x \in B(0,n) } \{x \in \mathscr{C}_\infty, R_x \geq n^\gamma \}$. By a union bound we see that 
\begin{equation}
\sum_{n = 0}^\infty \mathbb{Q}[D_n] \leq \sum_{n= 0}^\infty cn^d \exp\left( - c' n^{c_1\gamma} \right) < \infty,
\end{equation}
so by the Borel-Cantelli lemma, $\left(\bigcup_{ x \in B(0,n) } \{x \in \mathscr{C}_\infty, R_x \geq n^\gamma \} \right)^c = \bigcap_{x \in B(0,n)} (\{x \notin \mathscr{C}_\infty\} \cup \{ R_x < n^\gamma \})$ happens for all $n \geq N_0^{\textnormal{reg}}(\mu)$ for $\mu$ in a set of full $\mathbb{Q}$-measure, and the claim follows. 
\end{proof}

Finally, we also introduce some convenient notation concerning stochastic integrability motivated from the theory of quantitative stochastic homogenization (see in particular~\cite[Appendix A]{armstrong2019quantitative}), which will be helpful for some quantitative ergodic estimates in Section~\ref{s.ProofMainProp}. Given a real random variable $Y$ defined on some probability space $(E, \mathcal{E}, \widetilde{\mathbb{P}})$, and $s,\theta \in (0,\infty)$, we write
\begin{equation}
Y \leq \mathcal{O}_s(\theta) \text{ if and only if } \widetilde{\mathbb{E}}\left[\exp\left(\left(\frac{Y}{\theta} \right)_+^s\right) \right] \leq 2
\end{equation}
(with $\widetilde{\mathbb{E}}$ denoting the expectation under $\widetilde{\mathbb{P}}$). Note that $Y \leq \mathcal{O}_s(\theta)$ implies that for $y \geq 0$, one has 
\begin{equation}
\label{eq:ExpMarkov}
\widetilde{\mathbb{P}}[Y \geq \theta y] \leq 2 \exp\left( -y^s \right).
\end{equation}
 Also note that for any $s > 0$, there exists a constant $c_2(s)$ such that for $\theta_1, \theta_2 \in (0,\infty)$ and random variables $Y_1$ and $Y_2$, one has the implication
\begin{equation}
\label{eq:OsAdditionProperty}
Y_j \leq \mathcal{O}_s(\theta_j), \ j \in \{1,2\} \qquad \Rightarrow \qquad Y_1+Y_2 \leq \mathcal{O}_s\left(c_2(s)(\theta_1+\theta_2) \right),
\end{equation}
see~\cite[Lemma A.4]{armstrong2019quantitative} or~\cite[(24)]{dario2021quantitative}.

\section{Fractional moment method and proof of Theorem~\ref{thm:MainTheorem}}
\label{s.FracMoments}

In this section the proof of the main result, Theorem~\ref{thm:MainTheorem}, is presented. The proof is based on the change of measure technique and a fractional moment calculation which was introduced in~\cite{lacoin2010new} and recently used in a similar form in~\cite{cosco2021directed}. As a main step, we show below that for $\mathbb{Q}_0$-a.e.~$\mu \in \{0 \in \mathscr{C}_\infty \}$, the expectation $\mathbb{E}[W_{n,\mu}^\alpha]$ of the normalized partition function defined in~\eqref{eq:NormPartFct} raised to some power $\alpha \in (0,1)$ converges to zero as $n \rightarrow \infty$. For this argument a pivotal ingredient is to ensure that the random walk visits a ``tube'' of length $\varepsilon \log(n)$ before time $n$ with high probability for $\mathbb{Q}_0$-a.e.~$\mu$, and spends a time of order $(\varepsilon \log(n))^3$ there (with $\varepsilon > 0$ to be determined later). This part of the proof is postponed to Section~\ref{s.ProofMainProp}. \medskip

Let $\{e_1,...,e_d\}$ denote the canonical basis of $\mathbb{R}^d$. A set of edges of the form 
\begin{equation}
T_{x,L} = \Big\{\{x,x+e_1\},\{x+e_1,x+2e_1\},...,\{x+([ L ] -1)e_1,x+[L] e_1\} \Big\}
\end{equation}
 will be called an \textit{open tube} (in direction $e_1$) of length $L \in (0,\infty)$ based at $x \in \mathbb{Z}^d$ if $\mu_{e} = 1$ for all $e \in T_{x,L}$, and $\mu_e = 0$ for all adjacent edges in coordinate directions perpendicular to $e_1$ except for possibly those at $x$, and for $e = \{x+[L]e_1,x+([L]+1)e_1) \}$  (we use the convention that $T_{x,L} = \emptyset$ if $[L] = 0$, and such ``tubes'' of length $L < 1$ are always open).  More precisely, if $V(T_{x,L}) = \{x,x+e_1,...,x + [L]e_1\}$ denotes the vertex set of the tube $T_{x,L}$, the latter condition means that we require that $\mu_e = 0$ for every $e = \{z,z\pm e_j\} \in E(\mathbb{Z}^d)$, for $j \in \{2,...,d\}$ and $z \in V(T_{x,L}) \setminus \{x\}$, as well as $\mu_{\{x+[L]e_1,x+([L]+1)e_1 \}} = 0$. Let us already point out that with this definition, the vertex set of an open tube $T_{x,L}$ can only be visited by a simple random walk through $x$. We also consider the outer (edge) boundary $\partial_{\text{out}}T_{x,L}$ of the tube, which is defined as the set of edges with $\ell^1$-distance $1$ to $V(T_{x,L}) \setminus \{x\}$, formally
\begin{equation}
\partial_{\text{out}}T_{x,L} = \Big\{ e \in E(\mathbb{Z}^d) \, : \, d_{\ell^1}(e, V(T_{x,L}) \setminus \{x\}) = 1 \Big\}.
\end{equation} 
 
  We call an open tube $T_{x,L}$ \textit{good} if all edges in $\partial_{\text{out}}T_{x,L}$ are open.  An illustration is found in Figure~\ref{fig:tube}.

\begin{figure}[htbp]
\begin{center}
\includegraphics[scale=1]{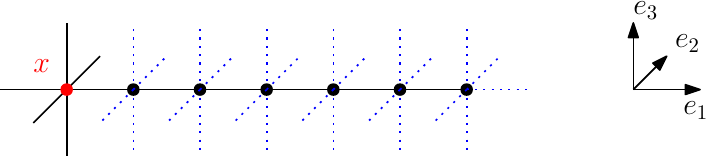}
\end{center}
\caption{An illustration of a situation in $d = 3$ in which $T_{x,L}$ is an open tube based at $x$ with $L = 6$. Here, dashed lines represent closed edges and full lines represent open edges. The tube $T_{x,L}$ is good if all edges adjacent to the dashed lines are open.}
\label{fig:tube}
\end{figure}

 In the following, we consider for $n \in \mathbb{N}$ and $\varepsilon > 0$ the event
  \begin{equation}
  \label{eq:SpendMuchTimeEventDef}
            \begin{minipage}{0.8\linewidth}\begin{center}
               $\mathcal{A}_n = \Big\{(X_k)_{k \in \{0,...,n\}}$ takes at least $[\varepsilon \log(n)]^3$ consecutive steps in an open tube of length $[\varepsilon \log(n)] \Big\}$,\end{center}
            \end{minipage}
        \end{equation}
which is defined for a fixed realization of $\mu \in \{0 \in \mathscr{C}_\infty\}$. More precisely,~\eqref{eq:SpendMuchTimeEventDef} means that there is an $x \in \mathscr{C}_\infty$ such that $T_{x,\varepsilon\log(n)} ( = T_{x,[\varepsilon\log(n)]})$ is an open tube and $j \in \{0,...,n\}$ with $j + [\varepsilon \log(n)]^3 \leq n$ such that $\{X_j,X_{j+1},...,X_{j + [\varepsilon\log(n)]^3} \} \subseteq V(T_{x,\varepsilon \log(n)})$.  A pivotal point in the construction below is that as $n$ tends to infinity, $\mathcal{A}_n$ becomes typical for $\mathbb{Q}_0$-a.e. realization of $\mu$, which we state as a proposition below.

\begin{proposition}
\label{prop:NoTubeAvoidance}
For $\varepsilon > 0$ small enough one has that for $\mathbb{Q}_0$-a.e.~realization $\mu \in \{0 \in \mathscr{C}_\infty \}$, 
\begin{equation}
\label{eq:TubeAvoidanceUnlikely}
\lim_{n \rightarrow \infty} P_{0,\mu}\left[ \mathcal{A}_n \right] = 1.
\end{equation}
\end{proposition}
The proof of Proposition~\ref{prop:NoTubeAvoidance} is the most technical aspect of this article, and we postpone it to Section~\ref{s.ProofMainProp}. We will now show how Theorem~\ref{thm:MainTheorem} can be obtained given the validity of Proposition~\ref{prop:NoTubeAvoidance}.

\begin{proof}[Proof of Theorem~\ref{thm:MainTheorem}] 
Fix a realization $\mu \in \{0 \in \mathscr{C}_\infty\}$ for which~\eqref{eq:TubeAvoidanceUnlikely} holds. Note that upon using~\eqref{eq:NormPartFctRewritten}, one can write
\begin{equation}
W_{n,\mu} = E_{0,\mu}\left[\exp\left(\sum_{i = 1}^n ( \beta \omega(i,X_i) - \Lambda(\beta)) \right) \right].
\end{equation}
We then choose $\alpha \in (0,1)$ and find that
\begin{equation}
\label{eq:MainProofStartingPoint}
\begin{split}
\mathbb{E}\left[W_{n,\mu}^\alpha \right] & \leq  \mathbb{E}\left[ E_{0,\mu}\left[\exp\left(\sum_{i = 1}^n ( \beta \omega(i,X_i) - \Lambda(\beta)) \right) \mathbbm{1}_{\mathcal{A}_n} \right]^\alpha \right] \\
& + \mathbb{E}\left[ E_{0,\mu}\left[\exp\left(\sum_{i = 1}^n ( \beta \omega(i,X_i) - \Lambda(\beta)) \right) \mathbbm{1}_{\mathcal{A}_n^c} \right]^\alpha \right],
\end{split}
\end{equation}
where we used that $(a+b)^\alpha \leq a^\alpha + b^\alpha$ for every $a,b > 0$. For the second term in the display above, one can use Jensen's inequality for the concave function $x \mapsto x^\alpha$ on $[0,\infty)$, which yields together with the definition~\eqref{eq:LambdaDef} of $\Lambda$ that
\begin{equation}
\label{eq:FirstTermSmall}
\mathbb{E}\left[ E_{0,\mu}\left[\exp\left(\sum_{i = 1}^n ( \beta \omega(i,X_i) - \Lambda(\beta)) \right) \mathbbm{1}_{\mathcal{A}_n^c} \right]^\alpha \right] \leq P_{0,\mu}[\mathcal{A}_n^c]^\alpha \stackrel{\eqref{eq:TubeAvoidanceUnlikely}}{=} o_n(1),
\end{equation}
as $n \rightarrow \infty$. Let 
\begin{equation}
\mathcal{T} = \Big\{x \in \mathscr{C}_\infty \, : \, \text{$T_{x,\varepsilon \log(n)}$ is an open tube} \Big\}
\end{equation}
stand for the set of points in the infinite cluster that have an open tube of length $\varepsilon \log(n)$ in direction $e_1$ attached to them. We use the decomposition
\begin{equation}
\begin{split}
\mathcal{A}_n & \subseteq \bigcup_{j = 0}^{n - [\varepsilon \log(n)]^3 } \bigcup_{x \in B(0,n) \cap \mathcal{T}} \mathcal{A}_{n,j,x}, \ \text{where } \\
\mathcal{A}_{n,j,x} & = \{ \{ X_k \, : k \in \{j,...,j+[\varepsilon \log(n)]^3  \} \} \subseteq V(T_{x,\varepsilon \log(n)}), X_j = x  \}, \ 0 \leq j \leq n, x \in B(0,n) \cap \mathcal{T},
\end{split}
\end{equation}
Indeed, by definition of $\mathcal{A}_n$, we can find some $0 \leq j \leq n - [\varepsilon \log(n)]^3$ and an open tube $T_{x,\varepsilon \log(n)}$ with $x \in \mathscr{C}_\infty$ such that the walk $(X_k)_{k \geq 0}$ takes $[\varepsilon \log(n)]^3$ steps in the tube after entering it at time $j$. Returning to~\eqref{eq:MainProofStartingPoint}, we find that 
\begin{equation}
\label{eq:ProofIntermediateStep}
\begin{split}
\mathbb{E}\left[W_{n,\mu}^\alpha \right] &  \stackrel{\eqref{eq:FirstTermSmall}}{\leq} \sum_{j = 0}^{n- [\varepsilon \log(n)]^3} \sum_{x \in B(0,n) \cap \mathcal{T}} \mathbb{E}\left[ E_{0,\mu}\left[\exp\left(\sum_{i = 1}^{n} ( \beta \omega(i,X_i) - \Lambda(\beta)) \right) \mathbbm{1}_{\mathcal{A}_{n,j,x}} \right]^\alpha \right] \\
& +o_n(1).
\end{split}
\end{equation}
Note that in the above sum, we only retain those contributions in which $T_{x,\varepsilon\log(n)}$ is an open tube. For such $x \in B(0,n)$, and $j \in \{0,...,n - [\varepsilon \log(n)]^3\}$ we introduce the set 
\begin{equation}
\mathcal{C}_{n,j,x} = \{j,j+1,...,j+[\varepsilon \log(n)]^3 \} \times V(T_{x,\varepsilon \log(n)}) (\subseteq \mathbb{N}_0 \times \mathbb{Z}^d),
\end{equation}
and define 
\begin{equation}
\label{eq:ProofIntermediateStepSmall}
W_{n,\mu,j,x} = E_{0,\mu}\left[\exp\left(\sum_{i = 1}^n ( \beta \omega(i,X_i) - \Lambda(\beta)) \right) \mathbbm{1}_{\mathcal{A}_{n,j,x}} \right].
\end{equation}
Clearly, one has for large enough $n$ 
\begin{equation}
\label{eq:HowManyElementsinC}
|\mathcal{C}_{n,j,x}| \leq C(\varepsilon \log(n))^4.
\end{equation} 
We will now use the aforementioned change of measure argument to bound the terms~\eqref{eq:ProofIntermediateStepSmall} contributing to~\eqref{eq:ProofIntermediateStep}. To this end, we define for $n \in \mathbb{N}$ the quantity
\begin{equation}
\label{eq:delta_n_def}
\delta_n = \frac{1}{(\log(n))^{\frac{7}{4}}\vee 1},
\end{equation}
and consider (similarly as in~\cite[Section 3]{lacoin2010new}) the tilted measures
\begin{equation}
\frac{\mathrm{d}\widetilde{\mathbb{P}}_{n,j,x} }{\mathrm{d}\mathbb{P}} = \prod_{(i,y) \in \mathcal{C}_{n,j,x}} \exp\Big(-\delta_n \omega(i,y) - \Lambda(-\delta_n)\Big).
\end{equation}
Under $\widetilde{\mathbb{P}}_{n,j,x}$, the random variables $\{\omega(i,y) \, : \, i \in \mathbb{N}_0, y \in \mathbb{Z}^d \}$ are independent with mean $-\delta_n \mathbbm{1}_{(i,y) \in \mathcal{C}_{n,j,x}}(1 + o_n(1))$ and variance $1+o_n(1)$ as $n \rightarrow \infty$, using~\eqref{eq:UsefulExpressionLambda}. We can now conclude the proof similarly as in~\cite{cosco2021directed}. By H\"older's inequality, we have
\begin{equation}
\label{eq:Holder}
\mathbb{E}[W_{n,\mu,j,x}^\alpha] = \widetilde{\mathbb{E}}_{n,j,x}\left[\frac{\mathrm{d} \mathbb{P} }{\mathrm{d} \widetilde{\mathbb{P}}_{n,j,x} } W_{n,\mu,j,x}^\alpha  \right] \leq \widetilde{\mathbb{E}}_{n,j,x}\left[\left(\frac{\mathrm{d} \mathbb{P} }{\mathrm{d} \widetilde{\mathbb{P}}_{n,j,x} } \right)^{\frac{1}{1-\alpha}} \right]^{1-\alpha} \widetilde{\mathbb{E}}_{n,j,x}\left[W_{n,\mu,j,x} \right]^\alpha,
\end{equation}
where we have denoted the expectation under $\widetilde{\mathbb{P}}_{n,j,x}$ by $\widetilde{\mathbb{E}}_{n,j,x}$. The first expression is bounded as follows for large enough $n$:
\begin{equation}
\begin{split}
\label{eq:FirstExpSmall}
\mathbb{E}\left[\left(\frac{\mathrm{d} \mathbb{P} }{\mathrm{d} \widetilde{\mathbb{P}}_{n,j,x} } \right)^{\frac{\alpha}{1-\alpha}} \right]^{1-\alpha} & \leq \exp\left( \frac{|\mathcal{C}_{n,j,x}|}{2 } \frac{\alpha \delta_n^2}{1-\alpha}(1+o_n(1))  \right) \\
& \stackrel{\eqref{eq:HowManyElementsinC}, \eqref{eq:delta_n_def}}{ \leq} \exp\left( C' \varepsilon^4 (\log(n))^{\frac{1}{2}} \frac{\alpha}{1-\alpha}(1+o_n(1)) \right).
\end{split}
\end{equation}
Moreover, we have for $(i,y) \in \mathcal{C}_{n,j,x}$ that 
\begin{equation}
\begin{split}
\exp(-\Lambda(\beta)) \cdot \widetilde{\mathbb{E}}_{n,j,x}\left[\exp(\beta \omega(i,y))\right] & = \exp\left(\Lambda(\beta-\delta_n) - \Lambda(\beta) - \Lambda(-\delta_n) \right) \\
& = \exp\left(-\Lambda'(\beta)\delta_n ( 1 + o_n(1))\right),
\end{split}
\end{equation}
as $n \rightarrow \infty$, therefore
\begin{equation}
\label{eq:SecondExpSmall}
\begin{split}
\widetilde{\mathbb{E}}_{n,j,x}\left[ W_{n,\mu,j,x} \right] & = E_{0,\mu}\left[\exp\left(-\Lambda'(\beta)\delta_n(1+o_n(1)) |\{k \, : \, (k,X_k) \in \mathcal{C}_{n,j,x} \} | \right) \mathbbm{1}_{\mathcal{A}_{n,j,x}} \right] \\
& \leq \exp\left(-C \Lambda'(\beta) \delta_n (1+o_n(1))(\varepsilon \log(n))^3 \right) \\
& \stackrel{\eqref{eq:delta_n_def}}{=} \exp\left(-C \Lambda'(\beta)  (1+o_n(1))\varepsilon^3 (\log(n))^{\frac{5}{4}} \right).
\end{split}
\end{equation}
We can now combine~\eqref{eq:ProofIntermediateStep},~\eqref{eq:Holder},~\eqref{eq:FirstExpSmall}, and~\eqref{eq:SecondExpSmall}, which yields
\begin{equation}
\begin{split}
\mathbb{E}\left[W_{n,\mu}^\alpha \right] &  \leq Cn^{d+1} \exp\left(C'\varepsilon^4(\log(n))^{\frac{1}{2}}\frac{\alpha(1+o_n(1))}{1-\alpha}-C \alpha \Lambda'(\beta)  (1+o_n(1))\varepsilon^3 (\log(n))^{\frac{5}{4}}  \right) \\
& \rightarrow 0, \qquad \text{ as } n \rightarrow \infty,
\end{split}
\end{equation}
having also used that by our assumptions in~\eqref{eq:AssumptionsExpVar} and below~\eqref{eq:LambdaDef}, $\Lambda'(\beta) > 0$ for all $\beta > 0$. Therefore, and since $(W^\alpha_{n,\mu})_n$ is uniformly integrable, we see that necessarily $W_{\infty,\mu} = 0$ $\mathbb{P}$-a.s., which implies $\beta_c = 0$.
\end{proof}

\section{Proof of Proposition~\ref{prop:NoTubeAvoidance}}
\label{s.ProofMainProp}

In this section we present the proof of Proposition~\ref{prop:NoTubeAvoidance}. To this end, we will first argue that for typical realizations of the percolation configuration, the infinite cluster $\mathscr{C}_\infty$ contains enough open tubes of logarithmic length (see Lemma~\ref{lem:ManyTubes} below) in every mesoscopic box of radius $n^{\frac{1}{4}}$ centered at $x \in B(0,n)$, provided that $n$ is large enough (due to the bipartite nature of $\mathbb{Z}^d$, we need to argue that this remains true for the parts of all mesoscopic boxes consisting of all points with  odd resp.~even parity). We then show that as $n$ tends to infinity, a random walk starting in some point $y \in B(0,n) \cap \mathscr{C}_\infty$ has a good chance to enter one of these tubes and to spend at least $[\varepsilon \log(n)]^3$ time steps there within its first $[\sqrt{n}]$ steps (see Lemma~\ref{lem:QuenchedEstimate} below), from which the claim in Proposition~\ref{prop:NoTubeAvoidance} is obtained. 
For the next statement, recall the definition of (good) open tubes in the beginning of Section~\ref{s.FracMoments}.

\begin{lemma}
\label{lem:ManyTubes}
Let $\delta > 0$. There exists $\varepsilon_0 = \varepsilon_0(p,\delta) > 0$ and for $\mathbb{Q}_0$-a.e.~$\mu \in \{0 \in \mathscr{C}_\infty \}$, there exists $N_{\textnormal{tube}} = N_{\textnormal{tube}}(\mu,\delta) < \infty$, such that for all $n \geq N_{\textnormal{tube}}$, one has
\begin{equation}
\label{eq:ManyTubes}
\begin{split}
&\textit{For every $z \in B(0,n)$, the sets $B_o(z,n^{\frac{1}{4}})$ and $B_e(z,n^{\frac{1}{4}})$ each contain at least} \\
& \textit{$c(p) n^{\frac{d-\delta}{4}}$ points $x \in \mathscr{C}_\infty$ for which $T_{x,\varepsilon \log(n)}$ is a good open tube},
 \end{split}
\end{equation}
whenever $\varepsilon \in (0,\frac{1}{4} \varepsilon_0)$.
\end{lemma}
Before we present the proof of Lemma~\ref{lem:ManyTubes}, we briefly give a heuristic outline of the main idea. 
Note that by the spatial ergodic theorem (see, e.g.,~\cite[p.~205]{krengel1985ergodic}), we have that
\begin{equation}
\label{eq:SpatialErgodic}
\lim_{n \rightarrow \infty} \frac{1}{|B(0,n)|} \sum_{x \in B(0,n)} \mathbbm{1}_{\{x \in \mathscr{C}_\infty \}} = \theta(p) (> 0), \qquad \mathbb{Q}\text{-a.s.},
\end{equation}
demonstrating that the infinite cluster has a non-vanishing density in large boxes. Moreover, the probability that at some point $x \in \mathbb{Z}^d$, $T_{x,\varepsilon \log(n)}$ is a good open tube is given by $\frac{1}{n^{c(p) \varepsilon}}$, meaning that in essence, we can hope for roughly $\theta(p) \cdot n^{d/4 - c(p)\varepsilon}$ many open tubes in a box of radius $n^{\frac{1}{4}}$. For our purposes, we essentially need some quantitative rate of convergence in~\eqref{eq:SpatialErgodic} and also include the existence of (good) open tubes of logarithmic size in our analysis. To that end, we will adapt the proof of the quantitative convergence provided in~\cite[Appendix A]{dario2021quantitative}, which relies on the exponential Efron-Stein inequality, see~\cite[Proposition 2.2]{armstrong2017optimal}.

\begin{proof}[Proof of Lemma~\ref{lem:ManyTubes}]
Let $n \in \mathbb{N}$, $\varepsilon > 0$ and consider the expression
\begin{equation}
\label{eq:ZDef}
\mathcal{Z}_n(\varepsilon) = \frac{1}{|B_o(0,n)|} \sum_{x \in B_o(0,n)} \mathbbm{1}_{\{x \in \mathscr{C}_\infty, \text{$T_{x,\varepsilon \log(n)}$ is a good open tube}\}}.
\end{equation}
We will be interested in deviations of $\mathcal{Z}_n(\varepsilon)$. Let us remark that analogous considerations as below work in the same way for $\mathcal{Z}_n'(\varepsilon) = \frac{1}{|B_e(0,n)|} \sum_{x \in B_e(0,n)} \mathbbm{1}_{\{x \in \mathscr{C}_\infty, \text{$T_{x,\varepsilon \log(n)}$ is a good open tube}\}}$, and we will focus on~\eqref{eq:ZDef} in what follows. We also define
\begin{equation}
\theta'(n,p,\varepsilon) = \mathbb{Q}\left[x \in \mathscr{C}_\infty, \text{$T_{x,\varepsilon \log(n)}$ is a good open tube} \right],
\end{equation}
which does not depend on $x \in \mathbb{Z}^d$ due to translation invariance. Moreover, note that $\{\text{$T_{x,\varepsilon \log(n)}$ is a good open tube}\}$ is the intersection of the independent events 
\begin{equation}
\begin{split}
A_{1,x} & = \{\text{$T_{x,\varepsilon \log(n)}$ is an open tube}\}, \\
A_{2,x} & = \{\mu_e = 1 \, : \, e \in \partial_{\text{out}}T_{x,\varepsilon \log(n) } \} \\
& ( = \{\text{all edges at $\ell^1$-distance $1$ to $V(T_{x,\varepsilon \log(n)})\setminus \{x\}$ are open} \}).
\end{split}
\end{equation}
Note that the definition of $A_{2,x}$ excludes the edge $\{x,x+e_1\}$, which is open if $A_{1,x}$ occurs (this choice ensures independence given $A_{2,x}$, since the events $A_{1,x}$ and $A_{2,x}$ depend on distinct edge sets). Similarly as in the proof of~\cite[Lemma 5.6]{cosco2021directed}, we note that $A_{2,x}$ and $\{x \in \mathscr{C}_\infty \}$ are increasing events, and moreover on $A_{2,x}$, the occurrence of $\{x \in \mathscr{C}_\infty\}$ does neither depend on the edges constituting the tube $T_{x,\varepsilon \log(n)}$, nor on the adjacent edges perpendicular to it, nor on $\{x+ [\varepsilon \log(n)]e_1,x+([\varepsilon \log(n)]+1)e_1 \}$. The second observation implies that $\{x \in \mathcal{C}_\infty\}$ and $A_{1,x}$ are conditionally independent given $A_{2,x}$, since on the probability space $(A_{2,x}, \mathcal{A}\vert_{A_{2,x}},\mathbb{Q}[\, \cdot \, |A_{2,x}])$ (with $\mathcal{A}\vert_{A_{2,x}} = \{B \cap A_{2,x} \, : \, B \in \mathcal{A} \}$), the events $A_{2,x} \cap \{x \in \mathcal{C}_\infty\}$ and $A_{2,x} \cap A_{1,x}$ depend on the values of $\mu_e$ on disjoint sets of edges, so
\begin{equation}
\label{eq:CondIndep}
\mathbb{Q}[A_{1,x} \cap \{x \in \mathcal{C}_\infty\} | A_{2,x}] = \mathbb{Q}[A_{1,x}| A_{2,x}] \cdot \mathbb{Q}[\{x \in \mathcal{C}_\infty\}| A_{2,x}].
\end{equation}
Using this together with the FKG inequality for increasing events (see~\cite{grimmett1999percolation}), we find that
\begin{equation}
\label{eq:TubeDensityLowerBound}
\begin{split}
\theta'(n,p,\varepsilon) & = \mathbb{Q}[A_{1,x} | \{ x \in \mathscr{C}_\infty\} \cap A_{2,x}] \mathbb{Q}[\{ x \in \mathscr{C}_\infty\} \cap A_{2,x}] \\
& \geq\mathbb{Q}[A_{1,x} | \{ x \in \mathscr{C}_\infty\} \cap A_{2,x}] \mathbb{Q}[x \in \mathscr{C}_\infty] \mathbb{Q}[A_{2,x}] \\
& = \mathbb{Q}[A_{1,x} | A_{2,x}] \mathbb{Q}[x \in \mathscr{C}_\infty] \mathbb{Q}[A_{2,x}] \\
& \geq \theta(p) \cdot p^{ c [\varepsilon \log(n)] } (1-p)^{c' [\varepsilon \log(n)]} \geq  \frac{\theta(p)}{n^{c_3(p) \varepsilon}},
\end{split}
\end{equation}
where we used in the third line above that~\eqref{eq:CondIndep} implies that $\mathbb{Q}[A_{1,x} | \{x \in \mathcal{C}_\infty\} \cap A_{2,x}] = \mathbb{Q}[A_{1,x} |A_{2,x}]$, and in the penultimate step the fact that $A_{1,x}$ is independent of $A_{2,x}$, and for both events we need to force at most $C[\varepsilon \log(n)]$ edges to be open or closed (note in passing that we could have similarly established the independence of $\{x \in \mathcal{C}_\infty \} \cap A_{2,x}$ and $A_{1,x}$ directly by arguing as above~\eqref{eq:CondIndep}). \medskip

Next, we consider the centered version of $\mathcal{Z}_n(\varepsilon)$ in~\eqref{eq:ZDef}, given by 
\begin{equation}
\overline{\mathcal{Z}}_n(\varepsilon) = \mathcal{Z}_n(\varepsilon) - \theta'(n,p,\varepsilon).
\end{equation}
Note that due to translation invariance, we have that $\mathbb{E}[\overline{\mathcal{Z}}_n(\varepsilon)] = 0$. We will prove that for $\varepsilon_0 = \varepsilon_0(p,\delta) \stackrel{\mathrm{def}}{=} \frac{\delta}{4c_3(p)}$, one has for every $n \in \mathbb{N}$ that
\begin{equation}
\label{eq:MainBoundTubeExistence}
\mathbb{Q}\left[|\overline{\mathcal{Z}}_n(\varepsilon)| \geq \frac{\theta(p)}{2n^{c_3(p)\varepsilon}}  \right] \leq C \exp\left(- c(p,\delta) n^{c_4(p,\delta)} \right),
\end{equation}
whenever $\varepsilon \in (0,\varepsilon_0)$. Let us first prove that~\eqref{eq:MainBoundTubeExistence} implies the statement of the lemma. To that end, consider random variables for $z \in \mathbb{Z}^d$
\begin{equation}
\begin{split}
N_z(\mu)  = \inf\{ & n_0 \in \mathbb{N} \, : \, \text{$B_o(z,n)$ contains at least $c_5(p) n^{d-\delta}$ points with $x \in \mathscr{C}_\infty$}  \\
& \text{such that $T_{x,4\varepsilon \log(n)}$ is a good open tube for all $n \geq n_0$} \}
\end{split}
\end{equation}
(with the usual convention $\inf \emptyset = \infty$), where we set $c_5(p) = \frac{1}{2}\theta(p)$. By translation invariance, the laws of $(N_z)_{z \in \mathbb{Z}^d_o}$ are identical, and the laws of $(N_z)_{z \in \mathbb{Z}^d_e}$ are identical as well. Now note that by~\eqref{eq:MainBoundTubeExistence} and the Borel-Cantelli lemma, we have for $\mathbb{Q}$-a.e.~$\mu$, and every $n$ large enough (depending on $\mu$) and $4\varepsilon < \varepsilon_0$ the inequality 
\begin{equation}
\label{eq:CloseToAverage}
|\overline{\mathcal{Z}}_n(4\varepsilon)| = |\mathcal{Z}_n(4\varepsilon) - \theta'(n,p,4\varepsilon)| \leq \frac{\theta(p)}{2n^{4c_3(p)\varepsilon}},
\end{equation}
and the latter implies that 
\begin{equation}
\label{eq:ManyGoodTubesInCube}
\begin{split}
\sum_{x \in B_o(0,n)} \mathbbm{1}_{\{x \in \mathscr{C}_\infty, \text{$T_{x,4\varepsilon \log(n)}$ is a good open tube}\}} & \geq |B_o(0,n)| \left(\theta'(n,p,4\varepsilon) - \frac{\theta(p)}{2n^{4c_3(p)\varepsilon}}\right) \\ 
&\stackrel{\eqref{eq:TubeDensityLowerBound}}{\geq} |B_o(0,n)| \frac{\theta(p)}{2n^{4c_3(p)\varepsilon}}.
\end{split} 
\end{equation}
In particular if $4\varepsilon < \varepsilon_0 (= \frac{\delta}{4c_3(p)})$, we see that $\mathbb{Q}$-a.s., $N_0(\mu) < \infty$, and we have the control
\begin{equation}
\begin{split}
\mathbb{Q}[N_0 > n_0] & =  \mathbb{Q}\bigg[\bigcup_{n = n_0}^\infty \Big\{\text{$B_o(0,n)$ contains less than $c_5(p)n^{d-\delta}$ points $x \in \mathscr{C}_\infty$} \\
& \qquad \text{such that $T_{x,4\varepsilon \log(n)}$ is a good open tube} \Big\} \bigg]  \\
& \leq \sum_{n = n_0}^\infty \mathbb{Q}\left[|\overline{\mathcal{Z}}_{n}(4\varepsilon)| \geq \frac{\theta(p)}{2n^{4c_3(p)\varepsilon}} \right]  \\
& \stackrel{\eqref{eq:MainBoundTubeExistence}}{\leq} C \sum_{n = n_0}^\infty\exp\left(-c(p,\delta) n^{c_4(p,\delta)} \right) \\
& \leq C'(p,\delta) \exp\left(-c'(p,\delta) n_0^{c_4(p,\delta)}\right), \qquad n_0 \in \mathbb{N},
\end{split}
\end{equation}
where a union bound was used in the second line together with the fact that~\eqref{eq:CloseToAverage} implies~\eqref{eq:ManyGoodTubesInCube}. The same bound also holds for $\mathbb{Q}[N_z > n_0]$ for any $z \in \mathbb{Z}^d$ (upon redefining the constants if necessary). Consider now the event $\bigcup_{z \in B(0,n)} \{N_z > \frac{1}{100} n^{\frac{1}{4}} \}$. Since
\begin{equation}
\sum_{n = 1}^\infty \mathbb{Q}\left[ \bigcup_{z \in B(0,n)} \{N_z > \frac{1}{100} n^{\frac{1}{4}} \}\right] \leq \sum_{n = 1}^\infty Cn^d \exp\left(-c n^{\frac{c_4}{4}} \right) < \infty,
\end{equation}
we infer (again by the Borel-Cantelli lemma) that for some $N_{\text{tube},o}(\mu) (< \infty\text{, $\mathbb{Q}$-a.s.})$, one has that $N_z \leq \frac{1}{100} n^{\frac{1}{4}}$ for every $z \in B(0,n)$ whenever $n \geq N_{\text{tube},o}$, which implies that every $B_o(z,k)$ contains at least $c_5(p)k^{d-\delta}$ points $x \in \mathscr{C}_\infty$ such that $T_{x,4\varepsilon \log(k)}$ is a good open tube whenever $k > \frac{1}{100} n^{\frac{1}{4}}$. In particular, every $B_o(z,n^{\frac{1}{4}})$ contains at least $c(p)n^{\frac{d-\delta}{4} }$ many points $x \in \mathscr{C}_\infty$ such that $T_{x,  \varepsilon \log(n)}$ is a good open tube. Since this argument can be repeated for $B_o(\cdot,\cdot)$ replaced by $B_e(\cdot,\cdot)$, we obtain $N_{\text{tube},e}(\mu) (< \infty\text{, $\mathbb{Q}$-a.s.})$, and so the claim follows for $n \geq N_{\text{tube}}(\mu) \stackrel{\text{def}}{=} N_{\text{tube},o}(\mu) \vee N_{\text{tube},e}(\mu)$.  \medskip

We now turn to the proof of~\eqref{eq:MainBoundTubeExistence}, which follows the outline of~\cite[Proposition 11]{dario2021quantitative}. 
To that end consider an independent copy of $\mu = \{\mu(e) \}_{e \in E(\mathbb{Z}^d)}$, denoted by $\widetilde{\mu}$ (by enlarging the probability space $(E(\mathbb{Z}^d), \mathcal{A}, \mathbb{Q})$ if necessary) and define for $e \in E(\mathbb{Z}^d)$ the environment $\{ \mu^e(e')\}_{e' \in E(\mathbb{Z}^d)}$ obtained by resampling the configuration at bond $e$ by setting 
\begin{equation}
\mu^e(e') = \mu(e') \mathbbm{1}_{\{e' \neq e \}} +  \widetilde{\mu}(e') \mathbbm{1}_{\{e' = e \}}, \qquad e' \in E(\mathbb{Z}^d).
\end{equation}
We let $\overline{\mathcal{Z}}_n^e(\varepsilon)$ stand for $\overline{\mathcal{Z}}_n(\varepsilon)$ evaluated for the configuration $\mu^e$, and drop the dependence on $\varepsilon$ to ease notation below. Next, let $\mathscr{C}_\infty^e$ stand for the infinite cluster in configuration $\mu^e$ and say that $T_{x,L}$ is an $e$-good open tube if $T_{x,L}$ is a good open tube in configuration $\mu^e$. Note that one has, $\mathbb{Q}$-a.s., that either $\mathscr{C}_\infty^e \subseteq \mathscr{C}_\infty$ or $\mathscr{C}_\infty \subseteq \mathscr{C}_\infty^e$. Observe that if $e \in E(B(0,3n))$, changing the configuration $\mu$ at bond $e$ can have two effects on $\overline{\mathcal{Z}}_n$, namely either a modification of the infinite cluster or the emergence or destruction of a good open tube. Note furthermore that the second effect pertains to at most $C \varepsilon \log(n)$ many points in $B(0,3n)$, namely those $x \in B(0,3n)$ for which the event that $T_{x,\varepsilon \log(n)}$ is a good open tube depends on the edge $e$, and is absent for $e \in E(\mathbb{Z}^d) \setminus E(B(0,3n))$, if $n$ is large enough. More precisely, we have for large enough $n$ that
\begin{equation}
\begin{split}
| \overline{\mathcal{Z}}_n^e - \overline{\mathcal{Z}}_n| & \leq \frac{1}{|B_o(0,n)|} \bigg\vert \sum_{x \in B_o(0,n)} \Big( \mathbbm{1}_{\{x \in \mathscr{C}_\infty^e, \text{$T_{x,\varepsilon \log(n)}$ is an $e$-good open tube}\}} \\
& \qquad \qquad \qquad \qquad - \mathbbm{1}_{\{x \in \mathscr{C}_\infty, \text{$T_{x,\varepsilon \log(n)}$ is a good open tube}\}}\Big) \bigg\vert \\
& \leq \frac{1}{|B_o(0,n)|} \bigg\vert \sum_{x \in B_o(0,n)} \mathbbm{1}_{ \{ \text{$T_{x,\varepsilon \log(n)}$ is an $e$-good open tube}\}} \Big( \mathbbm{1}_{\{x \in \mathscr{C}_\infty^e\}} - \mathbbm{1}_{\{x \in \mathscr{C}_\infty\}}\Big) \bigg\vert \\
& +\frac{1}{|B_o(0,n)|} \bigg\vert \sum_{x \in B_o(0,n)}
 \mathbbm{1}_{\{x \in \mathscr{C}_\infty\}}\Big(\mathbbm{1}_{ \{ \text{$T_{x,\varepsilon \log(n)}$ is an $e$-good open tube}\}} \\
 & \qquad \qquad \qquad \qquad -  \mathbbm{1}_{\{ \text{$T_{x,\varepsilon \log(n)}$ is a good open tube}\}}\Big) \bigg\vert \\
& \leq \frac{1}{|B_o(0,n)|} | (\mathscr{C}_\infty^e \triangle \mathscr{C}_\infty) \cap B_o(0,n) | + \frac{C \varepsilon \log(n) }{|B_o(0,n)|} \mathbbm{1}_{\{e \in E(B(0,3n))\}},
\end{split}
\end{equation}
where $\triangle$ denotes the symmetric difference between two sets. In particular, we see (using $(a+b)^2 \leq 2a^2+2b^2$ for $a,b \in \mathbb{R}$)
\begin{equation}
\label{eq:TwoTermsToBound}
\begin{split}
\sum_{e \in E(\mathbb{Z}^d)} (\overline{\mathcal{Z}}_n^e - \overline{\mathcal{Z}}_n)^2 & \leq  \frac{2}{|B_o(0,n)|^2}  \sum_{e \in E(\mathbb{Z}^d)} |(\mathscr{C}_\infty^e \triangle \mathscr{C}_\infty) \cap B_o(0,n)|^2  \\
& + \frac{2}{|B_o(0,n)|^2} \sum_{e \in E(B(0,3n))} (C \varepsilon \log(n))^2 \\
& \leq \frac{2}{|B_o(0,n)|^2}  \sum_{e \in E(\mathbb{Z}^d)} |(\mathscr{C}_\infty^e \triangle \mathscr{C}_\infty) \cap B_o(0,n)|^2 + \frac{C'\varepsilon^2 \log^2(n)}{n^d}.
\end{split}
\end{equation}
Now note that the random term in the last line of the display above can be bounded by~\cite[Lemma A.1]{dario2021quantitative} (see the calculation on p.623 of the same reference), giving 
\begin{equation}
\label{eq:Stochint1}
\frac{2}{|B_o(0,n)|^2}  \sum_{e \in E(\mathbb{Z}^d)} |(\mathscr{C}_\infty^e \triangle \mathscr{C}_\infty) \cap B_o(0,n)|^2 \leq \mathcal{O}_{\frac{d-1}{(3d+1)d}}\left(C(p)n^{-d}\right).
\end{equation}
The second term in the last line of~\eqref{eq:TwoTermsToBound} is deterministic and trivially fulfills the stochastic integrability bound 
\begin{equation}
\label{eq:Stochint2}
\frac{C'\varepsilon^2 \log^2(n)}{n^d} \leq \mathcal{O}_{\frac{d-1}{(3d+1)d}}\left(C n^{-(d-\delta)} \right)
\end{equation} 
for every $\delta > 0$ (where $C = C(\delta)$). By applying~\eqref{eq:OsAdditionProperty}, we obtain upon inserting~\eqref{eq:Stochint1} and~\eqref{eq:Stochint2} into~\eqref{eq:TwoTermsToBound} that
\begin{equation}
\sum_{e \in E(\mathbb{Z}^d)} (\overline{\mathcal{Z}}_n^e - \overline{\mathcal{Z}}_n)^2  \leq \mathcal{O}_{\frac{d-1}{(3d+1)d}}\left(C(p,\delta) n^{-(d-\delta)} \right).
\end{equation}  
One can now repeat the same arguments leading to the proof of~\cite[Proposition 11]{dario2021quantitative}, and we reproduce the main steps here for the convenience of the reader. We introduce the random variables
\begin{equation}
\begin{split}
\overline{\mathcal{Z}}_{n,e} & = \mathbb{E}_{\mathbb{Q}}\left[\overline{\mathcal{Z}}_n \big\vert \mathcal{A}(E(\mathbb{Z}^d) \setminus \{e\}) \right], \qquad e \in E(\mathbb{Z}^d), \text{ and} \\ 
\mathbb{V}[\overline{\mathcal{Z}}_n] & = \sum_{e \in E(\mathbb{Z}^d)} (\overline{\mathcal{Z}}_n - \overline{\mathcal{Z}}_{n,e})^2,
\end{split}
\end{equation}
where $\mathcal{A}(U) = \sigma( \mu_e \, : \, e \in U)$ for every $U \subseteq E(\mathbb{Z}^d)$. Here and in the following we denote by $\mathbb{E}_{\mathbb{Q}}$ the expectation under the probability measure $\mathbb{Q}$. We shall use the exponential Efron-Stein inequality alluded to earlier (\cite[Proposition 2.2]{armstrong2017optimal}), which states that for every $\beta \in (0,2)$,
\begin{equation}
\label{eq:ExpEfronStein}
\mathbb{E}_{\mathbb{Q}}\left[\exp(|\overline{\mathcal{Z}}_n|^\beta)  \right] \leq C(\beta) \mathbb{E}_{\mathbb{Q}}\left[\exp\left( (C(\beta)\mathbb{V}[\overline{\mathcal{Z}}_n])^{\frac{\beta}{2-\beta}}  \right) \right]^{\frac{2-\beta}{\beta}}.
\end{equation} 
Moreover, we have the implication 
\begin{equation}
\sum_{e \in E(\mathbb{Z}^d)} (\overline{\mathcal{Z}}_n^e - \overline{\mathcal{Z}}_n)^2  \leq \mathcal{O}_{\frac{d-1}{(3d+1)d}}\left(C(p,\delta) n^{-(d-\delta)} \right) \quad \Rightarrow \quad \mathbb{V}[\overline{\mathcal{Z}}_n] \leq \mathcal{O}_{\frac{d-1}{(3d+1)d}}\left( C(p,\delta) n^{-(d-\delta)}\right),
\end{equation}  
which follows from~\cite[Lemma 3.1]{gu2019efficient}.

We can then apply~\eqref{eq:ExpEfronStein} to obtain in the same way as in the proof of~\cite[Proposition 11]{dario2021quantitative} that for any $\delta > 0$, $C(p,\delta) \cdot \overline{\mathcal{Z}}_n \cdot n^{\frac{d-\delta}{2}} $ admits a small exponential moment, namely 
\begin{equation}
|\overline{\mathcal{Z}}_n| \leq \mathcal{O}_{c_6}\left( \frac{C(p,\delta)}{n^{\frac{d-\delta}{2}}} \right), \qquad \text{ i.e. } \qquad \mathbb{E}_{\mathbb{Q}} \left[\exp\left( \frac{|\overline{\mathcal{Z}}_n| }{C(p,\delta) n^{-\frac{d-\delta}{2} } } \right)^{c_6}\right] \leq 2
\end{equation}
(where $c_6 = \frac{2(d-1)}{3d^2+2d-1}$). The claim~\eqref{eq:MainBoundTubeExistence} now follows by applying the exponential Markov inequality. Indeed, 
\begin{equation}
\begin{split}
\mathbb{Q}\left[|\overline{\mathcal{Z}}_n| \geq \frac{\theta(p)}{2n^{c_3(p)\varepsilon}}  \right]  & \stackrel{\eqref{eq:ExpMarkov}}{\leq} 2 \exp\left( - \left( \frac{\theta(p) n^{\frac{d-\delta}{2} }}{2C(p,\delta) n^{c_3(p)\varepsilon}} \right)^{c_6} \right) \\
& = 2 \exp\left(-c'(p,\delta) n^{c_6\frac{d- \delta - 2c_3(p)\varepsilon}{2} } \right).
\end{split}
\end{equation} 
Since we chose $\varepsilon_0 = \frac{\delta}{4c_3(p)}$, this then finishes the proof.
 \end{proof}

\begin{remark}
As pointed out in~\cite[Remark 16]{dario2021quantitative}, the stochastic integrability (i.e.~the constant $c_6 = \frac{2(d-1)}{3d^2+2d-1}$) is suboptimal. However, the spatial scaling in the proof essentially confirms the informal discussion below the statement of Lemma~\ref{lem:ManyTubes}, namely that the number of (good) open tubes connected to $\mathscr{C}_\infty$ in a large box of radius $n$ is heavily concentrated around $\theta(p)\cdot n^{d-\delta}$ for arbitrarily small $\delta > 0$ (provided $\varepsilon$ is small enough, depending on $p$ and $\delta$), which would be the expected number if the occurrence of good open tubes was independent. 
\end{remark}

We have now proved that for typical realizations of $\mu \in \{0 \in \mathscr{C}_\infty\}$, there is a sizeable amount of (good) open tubes with starting points in both $B_o(z,n^{\frac{1}{4}}) \cap \mathscr{C}_\infty$ and $B_e(z,n^{\frac{1}{4}}) \cap \mathscr{C}_\infty$ for every $z \in \mathbb{Z}^d$ with $|z|_\infty \leq n$. The next lemma shows that for every large enough $n$, a random walk started at any point $y \in B(0,n)$ enters a tube of length $[\varepsilon \log(n)]$ and spends a time $[\varepsilon \log(n)]^3$ there within its first $[\sqrt{n}]$ steps with a probability decaying not faster than $\frac{1}{n^\eta}$, for arbitrarily small $\eta > 0$.

\begin{lemma}
\label{lem:QuenchedEstimate}
Let $\eta > 0$. There exists $\varepsilon > 0$ (depending on $\eta,p$) such that for any $y \in B(0,n) \cap \mathscr{C}_\infty$, and $\mathbb{Q}$-a.e.~$\mu \in \Omega_0$, we have
\begin{equation}
\label{eq:MeetingOneTube}
\begin{split}
& P_{y,\mu}[\text{$(X_k)_{k = 0,...,[\sqrt{n}]}$ takes at least $[\varepsilon \log(n)]^3$ consecutive steps } \\
& \qquad \qquad  \text{ in an open tube of length $[\varepsilon \log(n)]$}]  \geq \frac{c(\mu,\eta)}{n^\eta},
\end{split}
\end{equation}
for every $n \geq N_1(\mu,\eta)$ with $N_1(\mu,\eta) < \infty$.
\end{lemma}
\begin{proof} 
We can work with a fixed realization $\mu \in  \Omega_1$ of the percolation cluster such that both the claim of Lemma~\ref{lem:ShortMixingTimesCube} with $\gamma = \frac{1}{8}$ and~\eqref{eq:ManyTubes}  for some $\varepsilon_0 > 0$ are fulfilled. In particular (with $\delta > 0$ to be chosen later), for every $n \geq N_1(\mu,\delta) \stackrel{\text{def}}{=} N_0^{\text{reg}}(\mu) \vee N_{\text{tube}}(\mu,\delta)$, we have that $R_y < n^{\frac{1}{8}}$ for every $y \in \mathscr{C}_\infty \cap B(0,n)$ and there are $c(p)n^{\frac{d-\delta}{4}}$ open tubes $T_{x,\varepsilon \log(n)}$ with $x \in B_o(y,n^{\frac{1}{4}}) \cap \mathscr{C}_\infty$, i.e.~we have $|\mathcal{T} \cap B_o(y,n^{\frac{1}{4}})| \geq c(p)n^{\frac{d-\delta}{4}}$, where $\mathcal{T}$ denotes the set of $x \in \mathscr{C}_\infty$ such that $T_{x,\varepsilon \log(n)}$ is an open tube. Similarly, we also have $|\mathcal{T} \cap B_e(y, n^{\frac{1}{4}})| \geq c(p)n^{\frac{d-\delta}{4}}$. We then consider the event (for $n$ large enough and $\varepsilon < \frac{1}{4}\varepsilon_0$)
\begin{equation}
\begin{split}
& \mathcal{I}_n  = \{X_{[\tfrac{1}{4}\sqrt{n}]} \in \mathcal{T}, X_j \in V(T_{X_{[\tfrac{1}{4}\sqrt{n}]},\varepsilon\log(n)}) \text{ for } j \in \{[\tfrac{1}{4}\sqrt{n}],...., [\tfrac{1}{4}\sqrt{n}] + [\varepsilon \log(n)]^3\} \}
\end{split}
\end{equation}
which is contained in the event under the probability in~\eqref{eq:MeetingOneTube}. Note that
\begin{equation}
\label{eq:GoesToTube}
\begin{split}
P_{y,\mu}[X_{[\tfrac{1}{4}\sqrt{n}]} \in \mathcal{T}] & \geq \sum_{x \in \mathcal{T} \cap B_o(y,n^{\frac{1}{4}}) } P_{y,\mu}[X_{[\tfrac{1}{4}\sqrt{n}]} =x] + \sum_{x \in \mathcal{T} \cap B_e(y,n^{\frac{1}{4}}) } P_{y,\mu}[X_{[\tfrac{1}{4}\sqrt{n}]} =x] \\ 
& \stackrel{\eqref{eq:HKLB}}{\geq} (|\mathcal{T} \cap B_o(y,n^{\frac{1}{4}} )| \wedge |\mathcal{T} \cap B_e(y,n^{\frac{1}{4}} )|) \frac{c}{n^{\frac{d}{4}}} \exp\left(- c \frac{\sup_{x \in B(y,n^{\frac{1}{4}})} | x-y|^2}{\sqrt{n}} \right) \\
& \geq \frac{c}{n^{\frac{\delta}{4}}}.
\end{split}
\end{equation}
Suppose now that $T_{x,\varepsilon \log(n)}$ is an open tube with $x \in \mathscr{C}_\infty$. Then 
\begin{equation}
\label{eq:GoesToMiddle}
P_{x,\mu}[X_j = x + je_1, j \in \{1,..., [\varepsilon\log(n)/2]\}] \geq \frac{1}{2d}\left( \frac{1}{2}\right)^{\varepsilon \log(n)}.
\end{equation}
Note that conditionally on the event under the probability in~\eqref{eq:GoesToMiddle}, the random walk is located at a point with distance larger than $\frac{1}{4} \varepsilon \log(n)$ to the end points $x$ and $x + [\varepsilon \log(n)]e_1$ of the tube $T_{x,\varepsilon \log(n)}$. By Lemma~\ref{lem:DeviationExitTime}, we find that for any $z \in V(T_{x,\varepsilon \log(n)})$ with distance larger than $\frac{1}{4}\varepsilon \log(n)$ to these end points:
\begin{equation}
\label{eq:StaysInTube}
P_{z,\mu}[\{X_0,...,X_{[\varepsilon\log(n)]^3} \} \subseteq V(T_{x,\varepsilon \log(n)}) ] \geq \exp\left(-c \varepsilon \log(n) \right) = \frac{1}{n^{c\varepsilon}}.
\end{equation}
Upon combining~\eqref{eq:GoesToTube},~\eqref{eq:GoesToMiddle}, and~\eqref{eq:StaysInTube} and using the simple Markov property at the entrance time $[\frac{1}{4}\sqrt{n}]$ into $\mathcal{T}$ resp.~at time $[\varepsilon \log(n)/2]$, we see that
\begin{equation}
P_{y,\mu}[\mathcal{I}_n] \geq \frac{c}{n^{c'\varepsilon + \frac{\delta}{4}}}.
\end{equation}
The claim now follows by choosing $\delta = \eta$ and then letting $\varepsilon < \frac{3\eta}{4c'} \wedge \frac{1}{4} \varepsilon_0(p,\delta)$.
\end{proof}
We are now ready to prove Proposition~\ref{prop:NoTubeAvoidance}. This will be done by iterating the statement of Lemma~\ref{lem:QuenchedEstimate} at times which are multiples of $[\sqrt{n}]$. In each of the $[\sqrt{n}]$ steps between two such times, the walk has a good chance of entering a tube and spending an atypically large time there. 
\begin{proof}[Proof of Proposition~\ref{prop:NoTubeAvoidance}]

Let $\mu$ be an element of a subset of $\{ 0 \in \mathscr{C}_\infty\}$ of full $\mathbb{Q}_0$-measure such that the statement of Lemma~\ref{lem:QuenchedEstimate} is fulfilled for some small enough $\varepsilon > 0$, and let $n \geq N_1(\mu)$. We define the following events $\mathcal{A}_n^{(j)}$ for $j = 0,...,  [\frac{1}{2} \sqrt{n}]$: 
\begin{equation}
 \begin{minipage}{0.8\linewidth}\begin{center}
               $\mathcal{A}_n^{(j)} = \Big\{(X_k)_{k \in \{j[\sqrt{n}],...,(j+1)[\sqrt{n}]\}}$ takes at least $[\varepsilon \log(n)]^3$ consecutive steps in an open tube of length $[\varepsilon \log(n)] \Big\}$.\end{center}
            \end{minipage}
\end{equation}
Clearly, $\bigcup_{j = 0}^{[\frac{1}{2}\sqrt{n}]} \mathcal{A}_n^{(j)} \subseteq \mathcal{A}_n$, so we have the following upper bound for the probability of $\mathcal{A}_n^c$:
\begin{equation}
\label{eq:ToIterate}
P_{0,\mu}\left[\mathcal{A}_n^c \right] \leq P_{0,\mu}\left[\bigcap_{j = 0}^{[\frac{1}{2}\sqrt{n}]} (\mathcal{A}_n^{(j)})^c \right].
\end{equation}
Note that $\mathcal{A}_n^{(0)}$ coincides with the event under the probability in~\eqref{eq:MeetingOneTube}. Now for any $j \in \{0,1,...,[\frac{1}{2}\sqrt{n}]\}$, at any time $t_j = j \cdot [\sqrt{n}] \leq \frac{1}{2}n$, the random walk is located at some point $x \in \mathscr{C}_\infty \cap B(0,\frac{1}{2}n)$. Upon applying the simple Markov property at time at time $t_{[\frac{1}{2} \sqrt{n}]}$, we find that
\begin{equation}
\begin{split}
P_{0,\mu}\left[\mathcal{A}_n^c \right] & \leq P_{0,\mu}\left[\bigcap_{j = 0}^{[\frac{1}{2}\sqrt{n}]-1} (\mathcal{A}_n^{(j)})^c \right] \cdot \sup_{x \in \mathscr{C}_\infty \cap B(0,\frac{1}{2}n)} P_{x,\mu}\left[(\mathcal{A}_n^{(0)})^c \right] \\
&  \stackrel{\text{Lemma~\ref{lem:QuenchedEstimate}}}{\leq}  P_{0,\mu}\left[\bigcap_{j = 0}^{[\frac{1}{2}\sqrt{n}]-1} (\mathcal{A}_n^{(j)})^c \right] \cdot \left(1 - \frac{c(\mu,\eta)}{n^\eta}\right).
\end{split}
\end{equation}
By applying the simple Markov property iteratively at times $t_{[\frac{1}{2}\sqrt{n}]-1},t_{[\frac{1}{2}\sqrt{n}]-2},$ ... $t_{1}$, we find that
\begin{equation}
P_{0,\mu}\left[\mathcal{A}_n^c \right] \leq \left(1 - \frac{c(\eta,\mu)}{n^\eta} \right)^{[\frac{1}{2}\sqrt{n}]+1},
\end{equation}
which converges to $0$ provided that $\eta < \frac{1}{2}$.
\end{proof}

\begin{remark}
1) It is instructive to compare the behavior of directed polymers on the infinite cluster of Bernoulli percolation to the corresponding simple random walk. Indeed, by the results of~\cite{sidoravicius2004quenched} for $d \geq 4$ and the extension to $d\geq 2$ in~\cite{berger2007quenched,mathieu2007quenched}, the simple random walk on the infinite cluster fulfills a quenched invariance principle on $\mathbb{Q}_0$-a.e.~realization of the percolation configuration. It therefore behaves diffusively (recall that this is also the case for random polymers for $\beta < \beta_c(d)$ when $d \geq 3$ on the integer lattice $\mathbb{Z}^d$). In the situation of directed polymers, we see that for every non-zero value of the inverse temperature, one has strong localization of the corresponding trajectory in all dimensions $d \geq 2$. In other words: The presence of any (small) disorder in the graph already destroys the diffusive behavior exhibited by the polymer on the full lattice. \medskip

2) It seems plausible that the methods used here are pertinent to treat the case where Bernoulli percolation is replaced by some model with finite range dependence. The situation becomes less clear if we consider percolation models with long-range dependence, in particular in the case of algebraically decaying correlations, such as random interlacements, the vacant set of random interlacements or level-sets of Gaussian fields. We mention that certain properties of Bernoulli percolation (such as the regularity of chemical distances on the infinite cluster or a quenched invariance principle) are still valid for such strongly correlated models, see~\cite{drewitz2014chemical,procaccia2016quenched}. 
\medskip

3) In a different direction, one might also naturally wonder what can be said about \textit{very strong disorder}, see~\cite[Section 1.3]{cosco2021directed} in our context. Roughly speaking, very strong disorder occurs when $p(\beta) = \limsup_{n \rightarrow \infty} \frac{1}{n} \mathbb{E}[\log W_{n,\mu}] < 0$ (the existence of this limit when $\mu \in \{0 \in \mathcal{C}_\infty\}$ follows from~\cite[Proposition 1.8]{cosco2021directed}), and one can show that there exists $\beta_c = \overline{\beta}_c(\mu) \geq 0$ such that $p(\beta) = 0$ when $\beta < \overline{\beta}_c$, and $p(\beta) < 0$ when $\beta > \overline{\beta}_c$ (see again~\cite[Proposition 1.8]{cosco2021directed}). From~\cite[Proposition 1.9]{cosco2021directed}, we can infer that very strong disorder holds for large enough $\beta$ (possibly depending on $\mu$), for $\mathbb{Q}_0$-a.e.~$\mu \in \{0 \in \mathcal{C}_\infty\}$, provided that the law of $\omega(i,x)$ is unbounded from above. On the full lattice $\mathbb{Z}^d$, one knows that $\overline{\beta}_c = 0$ when $d \leq 2$ (see~\cite{comets2006majorizing,lacoin2010new}), and that $\overline{\beta}_c > 0$ when $d\geq 3$ (see~\cite{comets2017directed}), however it is currently not known whether $\overline{\beta}_c$ and $\beta_c$ coincide on $\mathbb{Z}^d$, $d\geq 3$. It is plausible that the methods developed in this paper might be helpful to study whether or not $\overline{\beta}_c = 0$ for $\mathbb{Q}_0$-a.e.~$\mu \in \{0 \in \mathcal{C}_\infty\}$ holds, and whether bounds for the behavior of $p(\beta)$ close to $0$ can be obtained. For both these questions, one progress might come from combining the techniques in~\cite[Sections 4, 7]{lacoin2010new} with (strengthened) versions of the controls on the underlying random walk on the cluster obtained in the present article.
\end{remark}

\subsection*{Acknowledgements}
The author wishes to thank Ofer Zeitouni for suggesting this problem and for numerous inspiring discussions on this topic, and the anonymous referee for helpful comments and suggestions.

\bibliographystyle{abbrv}
\newcommand{\noop}[1]{} \def\cprime{$'$}

\end{document}